\documentclass[12pt]{article}

\usepackage{amssymb}
\usepackage{amsthm}
\usepackage[top=2cm, bottom=2cm, left=3cm, right=2cm]{geometry}
\usepackage{amsfonts}
\usepackage[fleqn]{amsmath}
\usepackage{subeqnarray}
\usepackage{accents}
\usepackage[utf8]{inputenc}
\usepackage{ amssymb }
\usepackage[T1]{fontenc}
\usepackage{dsfont}
\usepackage{graphicx}
\usepackage{xcolor}

\usepackage{amsmath,amsfonts,amssymb,amsmath,amscd}
\usepackage[active]{srcltx}
\usepackage{graphicx,array}
\usepackage{amsthm}
\usepackage{amsbsy}
\usepackage{soul}

\newtheorem{theorem}{Theorem}[section]
\newtheorem{proposition}[theorem]{Proposition}
\newtheorem{corollary}[theorem]{Corollary}
\newtheorem{lemma}[theorem]{Lemma}
\newtheorem{algoritmo}[theorem]{Algorithm}

\theoremstyle{definition}

\begin{document}
\vspace{-5mm}
	
	\title{Equivariant mappings and invariant sets \\
	on Minkowski space}
\maketitle

\centerline{\scshape Miriam Manoel\footnote{Email address (corresponding author):
miriam@icmc.usp.br}}
{\footnotesize \centerline{Departamento de Matem\'atica, ICMC}
\centerline{Universidade de S\~ao Paulo} \centerline{13560-970 Caixa
Postal 668, S\~ao Carlos, SP - Brazil }}

\centerline{\scshape Leandro N. Oliveira \footnote{Email address:
leandro.oliveira@ufac.br}}
{\footnotesize \centerline{Centro de Ciências Exatas e Tecnol\'ogicas - CCET, UFAC}
\centerline{Universidade Federal do Acre} \centerline{69920-900 Rod. BR 364, Km 04, Rio Branco - AC, Brazil }}

\begin{abstract}
	In this paper we introduce the systematic study of invariant functions and equivariant mappings defined on  Minkowski space under the action of the Lorentz group. We adapt some known results from the orthogonal group acting on the Euclidean space to the Lorentz group acting on the Minkowski space. In addition, an algorithm is  given to compute generators of the ring of functions that are invariant under an important class of Lorentz subgroups, namely when these are generated by involutions, which is also useful to compute equivariants. Furthermore, general results on invariant subspaces of the Minkowski space are presented, with a characterization of invariant lines and planes in  the two lowest dimensions.   
\end{abstract}



\section{Introdution}

The interest of studying the effect of symmetries in a mathematical model is common sense. They frequently appear in the applications, in general as a result of the  geometry of the configuration domain, the modeling assumptions, the method to reducing to normal forms and so on.  In dynamical systems, certain observations  are tipically related to patterns of symmetry, namely degeneracy of solutions, high codimension bifurcations, unexpected stabilities, phase relations, synchrony in coupled systems and periodicity of solutions.  These phenomena, not expected in absence of symmetries, can be predicted once the inherent symmetries are taken into account in the model formulation. This observation has stimulated the great achievements in the theory of equivariant dynamical systems during the last decades. The set-up is based on group representation theory, once the set formed by these symmetries has a group structure.   We mention for example the works in \cite{ABDM, BM5, BM4, BM3} where algebraic invariant theory is used to derive the general form of equivariant mappings that define vector fields on real or complex Euclidean $n$-dimensional vector spaces under the action of subgroups of the orthogonal group ${\bf O}(n)$,  $n \geq 2$. A numerical analysis of such systems through computational programming has also been done, after the pionner works in \cite{Gat, Gat96, Sturm}. Steady-state bifurcation, Hopf bifurcation, classification of equivariant singularities among other subjects  have also been extensively studied in isolated dynamical systems or in networks of dynamical systems, for which we cite \cite{BM5, BM4, Buono, chossat, Lamb99, MR15, SGP, GS06, GS69} and hundreds of references therein.

The results in the present work go in a related but distinct direction. We use tools from group representation and from invariant theory to introduce a systematic study of symmetries of problems defined on the Minkowski space $\mathbb{R}^{n+1}_1$, $n \geq 1$. This is the $(n+1)$-dimensional real vector space $\mathbb{R}^{n+1}$ endowed with the pseudo inner product $\left< , \right> :\mathbb{R}^{n+1}_1\times\mathbb{R}^{n+1}_1\longrightarrow \mathbb{R}$ of signature $(n,1)$,
$$\left< x,y\right> =  x^{t}Jy = \sum_{i=1}^{n}x_{i}y_{i}-x_{n+1}y_{n+1}, $$
where $$J=\left(\begin{array}{cc}
I_{n} & 0 \\ 0 & -1 \end{array}\right),$$
also denoted by $I_{n,1}$, and $I_{n} $ denoting the identity matrix of order $n$. This generalizes the spacetime $\mathbb{R}^{4}_1$, which is the mathematical structure on which Einstein's theory  of relativity is formulated.   In Minkowski space, the spatio-temporal coordinates of different observers are related by Lorentz transformations, so  any laws for systems in Minkowski spacetime must accordingly be Lorentz invariant \cite{Janssen}. The group of symmetries is assumed to be a subgroup of the Lorentz group
$${\bf O}( n,1) =\left\{ A\in GL\left( n+1,\mathbb{R}\right) :A^{t}JA=J\right\},$$ formed by isometries of $\mathbb{R}^{n+1}_1$; here the superscript $t$ denotes transposition. We consider the standard action of (subgroups of) ${\bf O}(n,1)$ on $\mathbb{R}^{n+1}_1$, namely given by matrix multiplication. We use algebraic invariant theory and group representation to deduce the basic general results for the systematic study of mappings with symmetries defined on Minkowski space. We refer to \cite{Ergin} and \cite{Read} for example for potential applications of our results in the context with symmetries. 

Many previous results for the Euclidean case adapt for our context in a natural way. The first to mention is a way to obtain equivariant mappings on $\mathbb{R}^{n+1}_1$ under the action of a subgroup $\Gamma < {\bf O}( n,1)$ from invariant functions under the diagonal action of same group defined on a cartesian product $\mathbb{R}^{n+1}_1\times\mathbb{R}^{n+1}_1$  (Theorem~\ref{algoinvariante}).  We also adapt and extend  \cite[Theorem 3.2]{BM4}, regarding the construction of the general forms of functions invariant under a class of Lorentz subgroups, given by  Algorithms~\ref{algoritmo}, using Reynolds operators.

There are however two new structures here, namely  the pseudo inner product and the distinct group structure of the set of symmetries, leading to new facts. For example, for an arbitrary Lorentz subgroup $\Gamma$,  not all Minkowski subspace $W \subset \mathbb{R}^{n+1}_1$  admits its orthogonal subspace $W^\perp$ (with respect to the pseudo inner product) as an invariant complement under the action of $\Gamma$; for the Euclidean case see  \cite[Propostion XII, 2.1]{GS69}. In fact,  we proof that this holds, namely that $W \oplus W^{\perp} = \mathbb{R}^{n+1}_1$ with both  $\Gamma$-invariant,  if, and only if, $W$ is $\Gamma$-invariant and nondegenerate with respect to the pseudo inner product  (Proposition~\ref{complementonaodegenerado}).  In addition, we give  a sufficient condition for a degenerate subspace to admit an invariant complement (Proposition~\ref{complementoluz}).  In Section~\ref{sec:invariant subspaces}  we characterize the invariant subspaces for the lowest dimensions $n=1,2$ using, in the second case,  the singular value decomposition of elements of the Lorentz group. This classification is directly  related to the structure of the Lorentz group: let ${\bf SO}_0(n,1)$ denote the connected component of the identity
and consider the elements (which we denote following the usual notation in the literature, as in (\cite{Gallier} for example),
$$ \Lambda^{p}=\left(\begin{array}{cc} I_{n-1,1} & 0 \\ 0 & 1 \end{array}\right), \ \ \  \Lambda^{t}= J.  $$
We use the decomposition of ${\bf O}(n,1)$  as a semi-direct product,
\begin{eqnarray} \label{semi-direct product}
{\bf O}(n,1)={\bf SO}_{0}(n,1) \rtimes \left({\bf Z}_2( \Lambda^{t}) \times {\bf Z}_2(\Lambda^{p})\right),
\end{eqnarray}
or, in other words, as the disjoint union
\begin{eqnarray} \label{decomposicao}
{\bf O}(n,1)={\bf SO}_{0}(n,1) \dot{\cup}  \Lambda^{p}{\bf SO}_{0}(n,1) \dot{\cup}  \Lambda^{t}{\bf SO}_{0}(n,1) \dot{\cup}\Lambda^{pt}{\bf SO}_{0}(n,1),
\end{eqnarray}
where  $\Lambda^{pt}= \Lambda^{p} \Lambda^{t}$.  We then use this decomposition to identify conjugacy classes of Lorentz subgroups, once subgroups in distinct connected components are nonconjugate. We also give the type of each possible invariant subspaces, namely, as spacelike, timelike or lightlike subspaces, and recognize which are fixed-point subspaces of Lorentz subgroups.   \\

Here is what comes in the following sections. Section~\ref{sec:mappings} is devoted to  invariant functions and equivariant mappings under subgroups of ${\bf O}(n,1)$. The two main results are Theorem~\ref{algoinvariante}, used to compute equivariants, and the algorithm presented in Subsection~\ref{subsec:algorithm}, used to compute invariants under a class of groups generated by involutions. An example  in Subsection~\ref{subsec:example} illustrates both methods.  In Section~\ref{sec:invariant subspaces} we present general results about invariant subspaces and their orthogonal complements, with special attention to fixed-point subspaces. The subjects of Subsections~\ref{subsec:2d} and \ref{subsec:3d}  are the invariant subspaces in the Minkowski plane and Minkowski 3-dimensional space.

\section{Invariant functions, equivariant mappings} \label{sec:mappings}
The aim of this section is to give results on the construction of invariant functions and equivariant mappings under the action of a Lorentz subgroup $\Gamma$. We adapt some results from invariant theory on Euclidean space for the Minkowski space. We point out that the results are algebraic in nature, so they hold for functions and mappings defined on any subspace of $ W \subseteq \mathbb{R}^{n+1}_1$ as long as it is $\Gamma$-invariant, namely, $\gamma x \in W$ for all $\gamma \in \Gamma$, $x \in W$.  For simplification, from now on we assume the domain to be the whole $\mathbb{R}^{n+1}_1$.

A function $f: \mathbb{R}^{n+1}_1  \rightarrow \mathbb{R}$ is called $\Gamma$-invariant  if
$$f( \gamma x)=f(x), \quad \forall \gamma \in \Gamma, \quad \forall x \in \mathbb{R}^{n+1}_1.$$
The set of $\Gamma$-invariant functions is a ring, which we shall denote $\mathcal{I}(\Gamma)$. A map $g:\mathbb{R}^{n+1}_1 \to \mathbb{R}^{n+1}_1$ is $\Gamma$-equivariant  if  it commutes with the action of $\Gamma$, that is, if $$g(\gamma x)=\gamma g(x) \quad \forall \gamma \in \Gamma, \quad \forall x \in \mathbb{R}^{n+1}_1.$$
The set of $\Gamma$-equivariant maps is a module over the ring $\mathcal{I}(\Gamma)$, denoted here by ${\mathcal{M}}(\Gamma)$.

In the Euclidean context it is well-known that if $f: \mathbb{R}^n \to \mathbb{R}$ is invariant under the action of a subgroup of the orthogonal group  ${\bf O}(n)$, then the gradient $\nabla f:\mathbb{R}^n \to \mathbb{R}^n$ is an equivariant map under this action. The corresponding result for the Lorentz group ${\bf O}(n,1)$ is given in Proposition~\ref{gradienteequivariante}, which can be used as a starting point to find generators for the module of equivariants if generators of the ring of invariants are known:

\begin{proposition} \label{gradienteequivariante}
	Let $\Gamma$ be a subgroup of ${\bf O}(n,1)$ and $f: \mathbb{R}^{n+1}_1 \to \mathbb{R}$ $\Gamma$-invariant. Then $J \nabla f$ is a $\Gamma$-equivariant map.
\end{proposition}

\begin{proof}
	The equality $f(\gamma x)=f(x)$ implies  $\gamma^t \nabla(f(\gamma x))=\nabla(f(x))$. Now just multiply  both sides by $\gamma J$.
\end{proof}

The following result shows that a $\Gamma$-equivariant mapping $\mathbb{R}^{n+1}_1 \to \mathbb{R}^{n+1}_1$ can be given from a $\Gamma$-invariant function   $\mathbb{R}^{n+1}_1  \times \mathbb{R}^{n+1}_1 \to \mathbb{R}$ on the cartesian product, and vice-versa. The proof is constructive and follows the idea of the proof of  \cite[Theorem 6.8]{GS69}, providing a formula to construct one from the other.

\begin{theorem} \label{algoinvariante}
	Let $\Gamma$ be a subgroup of the Lorentz group ${\bf O}(n,1)$.  There is a one-to-one correspondence between $\Gamma$-equivariant mappings $\mathbb{R}^{n+1}_1 \to \mathbb{R}^{n+1}_1$ and $\Gamma$-invariant functions $\mathbb{R}^{n+1}_1 \times \mathbb{R}^{n+1}_1  \to \mathbb{R}$ under the diagonal action.
\end{theorem}

\begin{proof}  Given $g: \mathbb{R}^{n+1}_1 \to \mathbb{R}^{n+1}_1$, take
$f: \mathbb{R}^{n+1}_1 \times \mathbb{R}^{n+1}_1 \to \mathbb{R},$
\begin{eqnarray} \label{invariant on cartesian}
 f(x,y) = \left<g(x), y) \right>.
\end{eqnarray}
This implies that
\begin{eqnarray} \label{equivariant}
g(x)=J\left(d_y f \right)^{t}_{(x,0)}.
\end{eqnarray}
If $g$ is $\Gamma$-equivariant, then for the diagonal action of $\Gamma$ on $\mathbb{R}^{n+1}_1 \times \mathbb{R}^{n+1}_1$ we have, for all $\gamma \in \Gamma$,
	\begin{eqnarray*}
		f(\gamma (x,y)) = \left< g(\gamma x),\gamma y\right>=\left< \gamma g(x),\gamma y\right>=f(x,y).
	\end{eqnarray*}
Conversely, for any $\gamma \in \Gamma$, if $f$  is $\Gamma$-invariant, then, differenting  $f(\gamma x, \gamma y) =f(x,y)$ at both sides with respect to $y$ at $(x,0)$,
	$$\left(d_y f \right)_{(\gamma x,0)} \gamma =\left(d_y f \right)_{(x,0)}.$$
Taking the transpose,
	$$g(\gamma x) =J \left(\gamma ^{t} \right)^{-1}Jg(x),$$
	so the result holds, since  $J \left(\gamma ^{t} \right)^{-1} J=\gamma$.
\end{proof}

The next result is our method to find a set of generators for the module of equivariant mappings: 

\begin{corollary} \label{cor: main}
If $\{u_i, i=1, \ldots s\}$ is a set of  generators of the ring of $\Gamma$-invariant functions  $\mathbb{R}^{n+1}_1 \times \mathbb{R}^{n+1}_1  \to \mathbb{R}$ under the diagonal action, then
$$J \left(d_y u_i \right)_{(x,0)}^t, \ i=1, \ldots, s$$
form a set of generators of the module   ${\mathcal{M}}(\Gamma)$ over ${\mathcal{I}}(\Gamma)$ .
\end{corollary}
\begin{proof}
For $g \in {\mathcal{M}}(\Gamma)$, we take $f$ as in (\ref{invariant on cartesian}). By hypothesis,
	$$f(x,y) = h(u_1(x,y),\cdots,u_s(x,y)),$$
	for some $h:\mathbb{R}^s \to \mathbb{R}$. This implies, using (\ref{equivariant}), that
	$$g(x)=\sum_{i=0}^{s} \frac{\partial h}{\partial X_i}(u_1,\cdots,u_s)|_{y=0} \cdot J \left(d_y u_i \right)^{t}_{(x,0)},$$
where $X = (X_1, \ldots, X_s) \in \mathbb{R}^s$. So the result follows, since ${\partial h}/{\partial X_i}(u_1(x,0),\cdots,u_s(x,0)) \in {\mathcal{I}}(\Gamma)$, $i=1, \ldots s$.
\end{proof}
\noindent This result is illustrated  with an example, presented in Subsection~\ref{subsec:example}.

\subsection{An algorithm to compute invariants under subgroups containing involutions}  \label{subsec:algorithm}

For the following, recall that an involution is an invertible map which is its own inverse. Involutions in ${\bf O}(n,1)$ are order-2 matrices, also called generalized reflections.  The aim here is to present  Algorithm~\ref{algoritmo} which gives generators of the ring of invariant functions  under the class of Lorentz subgroups  $ \Gamma < {\bf O} (n,1)$ that are generated by  a finite set of involutions. More generally, we consider
\begin{eqnarray}  \label{eq:index2m}
 \Gamma = \Sigma  \rtimes \Delta,
 \end{eqnarray}
for $\Sigma < \Gamma$ a subgroup such that ${\cal I}(\Sigma)$ is finitely generated and for  $\Delta  = [ \delta_1, \ldots, \delta_m] < \Gamma $ generated by involutions,   $\delta_i^2 = I$, $i = 1, \ldots, m$, $m \geq 1$. More precisely, we deduce generators under the whole group from the generators under the subgroup $\Sigma$. This result generalizes \cite[Theorem 4.5]{BM5},  obtained for $m = 1$, in which case $\Gamma = \Sigma_1 \rtimes {\bf Z}_2(\delta)$. Recall the two Reynolds operators on the ring of invariants   $ R, S: {\cal I}(\Sigma_1) \to {\cal I}({\Sigma_1}), $
\begin{eqnarray} \label{eq:delta inv}
R(f)(x)=f(x)+f(\delta x), \ \ \ \ \  S(f)(x)=f(x)- f(\delta x).
\end{eqnarray}
It follows that if $\{u_1,u_2,\cdots,u_s\}$ is a set of generators of the subring ${\cal I}(\Sigma_1)$, then  the set
\begin{eqnarray} \label{eq:gen inv}
\{R(u_i),S(u_i)S(u_j), 1 \leq i,j \leq s\}
\end{eqnarray}
generates the ring ${\cal I}(\Sigma_1 \rtimes {\bf Z}_2(\delta) )$ (\cite[Theorem 3.2]{BM4}). 

The result has been obtained in \cite{BM4} to deduce general forms of relative invariants. We use it here in a different context, establishing an algorithm for the calculation of generators of invariant functions under the action of a group $\Gamma$ as given in (\ref{eq:index2m}). The idea is to compute the generators recursively, imposing the invariance under each $\delta_i$, $i=1, \ldots, m$, at each step. This follows from the fact that the equality  
$${\cal I}(\Sigma \rtimes {\bf Z}_2(\delta)) \ = \   {\cal I}(\Sigma) \cap {\cal I} ({\bf Z}_2(\delta)), $$
which is the foundation to obtain (\ref{eq:gen inv}), generalizes to
$$ {\cal I}(\Sigma \rtimes [\delta_1, \ldots, \delta_m]) \ = \   {\cal I}(\Sigma) \ \bigcap_{i=1}^{m}  \ {\cal I} ({\bf Z}_2(\delta_i)). $$
Obviously the equalities above for the semi-direct product  hold for a direct product as well. We then have: 

	\begin{algoritmo} \label{algoritmo}
	\rule{0cm}{0cm}\\[-7mm]
\begin{description}
\item[\sc Input:] \parbox[t]{12cm}{
$\cdot$ a set $\{u_1,u_2,\cdots,u_s\}$ of generators of the ring ${\cal I}(\Sigma)$\\
$\cdot$ a set of  involutions $\delta_i, \ i=1, \ldots, m$} \\
		
\item[\sc Output:] \parbox[t]{11.5cm}{ Generators of ${\cal I}(\Gamma)$, where $\Gamma = \Sigma \rtimes [\delta_1, \ldots, \delta_m]$ or $\Sigma \times [\delta_1, \ldots, \delta_m]$}\\

\item[\sc Procedure:]\rule{0cm}{0cm}\\		
	    $\cdot$ $\Sigma_0:= \Sigma$\\
	    $\cdot$ $u_{0i}  := u_i$, for $i=1, \ldots s$  \\
	    $\cdot$ $s_0 := s$ \\
	    for k from 1 to $m$ do \\
		 \hspace*{5mm} $\Sigma_{k} :=\Sigma_{k-1} \rtimes {\bf Z}_2(\delta_k)$\\
		 \hspace*{5mm} for $\delta$ in  (\ref{eq:delta inv}), do $\delta:= \delta_k$\\
		 \hspace*{1cm} compute the set in  (\ref{eq:gen inv})  from  $\{u_{kj}, j =1 \ldots, s_k\}$  \\
		 \hspace*{1cm} return  generators of ${\cal I}(\Sigma_{k})$: $\{u_{k1}, \ldots, u_{k s_{k}} \}$\\
		  \hspace*{5mm}end\\
		end     \\
		return  generators of ${\cal I}({\Gamma})$: $\{u_{m1}, \ldots, u_{m s_{m}}\}$
		\end{description}
\end{algoritmo}
	
\subsection{An example} \label{subsec:example}

Consider the Lorentz subgroup $\Gamma < {\bf O}(3,1)$,
	$$\Gamma={\bf \widetilde{SO}}_{0}(3,1) \rtimes ({\bf Z}_2(\delta_1) \times {\bf Z}_2(\delta_2)),$$
	where
	\begin{eqnarray}
	{\bf \widetilde{SO}}_{0}(3,1) =\left\{ \xi=\left( \begin{array}{cc} \tilde{\xi} & 0 \\ 0 & I_2 \end{array} \right):  \ \tilde{\xi} \in {\bf {\bf SO}}(2) \right\}, \label{grupocompacto}
	\end{eqnarray}
	and for an arbitrary (but fixed) $\theta \in \mathbb{R},$
		$$\delta_1:= \left( \begin{array}{cccc} 1 & 0  & 0 & 0  \\ 
0 & 1 & 0 & 0 \\
0 & 0 &     \cosh \theta &  \sinh \theta \\
  0 & 0 &  - \sinh \theta & -   \cosh \theta  \end{array} \right),  \ \ 
 \delta_2:= \left( \begin{array}{cccc} 1 & 0  & 0 & 0  \\ 
 0 & 1 & 0 & 0 \\
 0 & 0 &    - \cosh \theta & -   \sinh \theta \\
   0 & 0 &   \sinh \theta &   \cosh \theta  \end{array} \right). $$
   We obtain a set of generators for the ring ${\cal I}(\Gamma)$ of $\Gamma$-invariant function $\mathbb{R}^{4}_1 \to \mathbb{R}$.
	The representation of ${\bf \widetilde {SO}}_{0}(3,1)$ is  isomorphic to the 4-dimensional standard representation of ${\bf SO}(2) \times \{I_2\}$; it then follows trivially that, in coordinates $(x_1,x_2,x_3,x_4)$ of $\mathbb{R}^{4}_1$  the set $\{u_1:= x_1^2+x_2^2,u_2:=x_3,u_3 := x_4\}$ generates the ring ${\cal I}({\bf \widetilde{SO}}_{0}(3,1))$. We now apply Algorithm~\ref{algoritmo}: for $k=1$, $\Sigma_1 ={\bf \widetilde{SO}}_{0}(3,1) \rtimes  {\bf Z}_2(\delta_1)$  and the Reynolds operators in (\ref{eq:delta inv}) are taken for $\delta = \delta_1$, which are denoted here by $R_{\delta_1},   \ S_{\delta_1}.$ The computation of the set  (\ref{eq:gen inv}) of generators of ${\cal I}(\Sigma_1)$ gives
	\begin{eqnarray*}
		R_{\delta_1}(u_{01})&=&x_1^2+x_2^2, \\
		R_{\delta_1}(u_{02})& =& \frac{1}{2}((   \cosh \theta+1) x_3+ \sinh \theta x_4), \\
		R_{\delta_1}(u_{03})& =&-\frac{1}{2}((   \cosh \theta-1) x_4+ \sinh \theta x_3),
	\end{eqnarray*}
and
\begin{eqnarray*}	S_{\delta_1}(u_{01})&=&0, \\
			S_{\delta_1}(u_{02})&=&-\frac{1}{2}((   \cosh \theta-1) x_3+ \sinh \theta x_4), \\
			S_{\delta_1}(u_{03})&=&\frac{1}{2}((   \cosh \theta+1) x_4+ \sinh \theta x_3).
\end{eqnarray*}
Observing that
$$R_{\delta_1}(u_{02})=((1-{   \cosh \theta})/{ \sinh \theta}) R_{\delta_1}(u_{03}), \ \  S_{\delta_1}(u_{02})=((1-{   \cosh \theta)}/{ \sinh \theta})S_{\delta_1}(u_{03}),$$
it follows that
 $\{R_{\delta_1}(u_{01}), R_{\delta_1}(u_{02}), S_{\delta_1}(u_{02})^2\}$ generates ${\cal I}({\bf \widetilde {SO}}_{0}(3,1)) \rtimes {\bf Z}_2(\delta_1))$.
 Alternatively, manipulating these generators, another generating set for this ring is given by
 $ \{u_{11}:= R_{\delta_1}(u_{01}),  u_{12}:=   R_{\delta_1}(u_{03}), u_{13}:= x_3^2-x_4^2\} ,$
  since
 $$\frac{1}{2}(   \cosh \theta-1)(R_{\delta_1}(u_{03})^2 - 2 S_{\delta_1}(u_{02})^2) = x_3^2 - x_4^2. $$

\noindent For $k=2$,  $\Sigma_2 = \Sigma_1 \rtimes  {\bf Z}_2(\delta_2),$  and the operators in (\ref{eq:delta inv}) are taken for $\delta = \delta_2$, which are denoted here by $R_{\delta_2},   \ S_{\delta_2}.$ The computation of the set  (\ref{eq:gen inv}) of generators of ${\cal I}(\Sigma_2)$ gives
	\begin{eqnarray*}
		R_{\delta_2}(u_{11})=x_1^2+x_2^2, \ \
		R_{\delta_2}(u_{12})=0, \ \
		R_{\delta_2}(u_{13})=x_3^2-x_4^2, \\
        S_{\delta_2}(u_{11}) = 0, \ \
		S_{\delta_2}(u_{12}) = (   \cosh \theta-1)x_4+ \sinh \theta x_3, \ \
		S_{\delta_2}(u_{13}) = 0.
	\end{eqnarray*}
It then follows that
$$u_{21}:= x_1^2+x_2^2, \ \ \ u_{22}:= x_3^2-x_4^2, \ \ \  u_{23}:= ((   \cosh \theta-1)x_4+
 \sinh \theta x_3)^2$$
generate  ${\cal I}(\Gamma)$. \\

We now compute a set of generators of the module ${\cal M}(\Gamma)$ of the $\Gamma$-equivariant mappings over the ring ${\cal I}(\Gamma)$.  We use Corollary~\ref{cor: main}.  \\

	For $(x,y)=(x_1,x_2,x_3,x_4,y_1,y_2,y_3,y_4) \in \mathbb{R}^4_1 \times \mathbb{R}^4_1,$ consider the diagonal action $\xi (x,y)=(\xi x, \xi y)$, for $\xi \in {\bf \widetilde{SO}}_{0}(3,1)$. The invariant ring under this subgroup is generated by  the polynomials
\begin{eqnarray} \label{eq:obvious}
u_{01}:= x_1^2+x_2^2, \ \ u_{02}:=x_3, \ \ u_{03}:=x_4, \ \ v_{01}:=y_1^2+y_2^2,  \ \  v_{02}:=y_3, \ \ v_{03}:=  y_4,
\end{eqnarray}
\begin{eqnarray} \label{eq:less obvious}
 w_{01}:=x_1 y_2-x_2 y_1, \ \ w_{02}=x_1 y_1+x_2y_2,
 \end{eqnarray}
The generators in (\ref{eq:obvious}) are the obvious ones.   In fact,
the computations above give trivially the generators of the ring
of  invariant functions on $\mathbb{R}^4_1 \times \mathbb{R}^4_1$ under
$\Gamma$, except the new ones coming from $w_1$, $w_2$ given in (\ref{eq:less obvious}), for which  $R_{\delta_1}(w_{01})) = w_{01}, \ R_{\delta_1}(w_{02}) = w_{02},$ and therefore $S_{\delta_1}(w_{01}) = 0, \ S_{\delta_1}(w_{02}) = 0. $
 It follows that
$$u_{11}:= x_1^2+x_2^2, \ u_{12}:= (   \cosh \theta-1) x_4+ \sinh \theta x_3, \ u_{13}:= x_3^2-x_4^2, $$
$$ v_{11}:= y_1^2+y_2^2, \  v_{12}:= (   \cosh \theta-1) y_4+ \sinh \theta y_3, \ v_{13}:= y_3^2-y_4^2,$$
$$	  w_{11}:= x_1 y_2-x_2 y_1, \  w_{12}:= x_1 y_1 + x_2 y_2 $$
are generators of ${\cal I}({\bf \widetilde{SO}}_{0}(3,1) \rtimes {\bf Z}_2(\delta_1)).$ Also,
$R_{\delta_2}(w_{11}) = w_{11}$ and  $R_{\delta_2}(w_{12}) = w_{12}$ and, therefore,
$S_{\delta_2}(w_{11}) = S_{\delta_2}(w_{12}) = 0$. Finally, $S_{\delta_2}(u_{12}) = u_{12}$ and $S_{\delta_2}(v_{12}) = v_{12}$, so $u_{12} v_{12}$ is an element in (\ref{eq:gen inv}).
We now use  (\ref{equivariant}) of Theorem~\ref{algoinvariante} which produces the zero map except that

	$$J(d_y w_{11})_{(x,0)}^t=( -x_2, x_1, 0, 0),$$
	$$J(d_y w_{12})_{(x,0)}^t=(x_1, x_2, 0, 0),$$
	$$J(d_y u_{12} v_{12})_{(x,0)}^t=(0, 0, \sinh \theta ,
	\cosh \theta -1) ((\cosh \theta-1) x_4+ \sinh \theta x_3),$$
forming	a system of generators of ${\cal M}(\Gamma)$ over ${\cal I}(\Gamma)$.

\section{Invariant subspaces}   \label{sec:invariant subspaces}

In this section we discuss about subspaces of $\mathbb{R}^{n+1}_1$  that are invariant under the action of a group $\Gamma < {\bf O}(n,1)$ and about existence of their invariant complement. For a given subspace $W \subseteq \mathbb{R}^{n+1}_1$, we shall say that a subspace is its  complement if they sum direct to give $\mathbb{R}^{n+1}_1$.
We start with some general results on invariant subspaces. In Subsections~\ref{subsec:2d} and \ref{subsec:3d} we characterize them for the lowest dimensions, in $\mathbb{R}^{2}_1$ and $\mathbb{R}^{3}_1$   respectively, and  classify according to their type, as space-, time- or lightlike subspaces. 

Recall that a nonzero vector $x \in \mathbb{R}^{n+1}_1$  is called \emph{spacelike}, \emph{timelike} or \emph{lightlike} if $\langle x,x \rangle >0$, $\langle x,x \rangle <0$ or $\langle x,x \rangle =0$, respectively.  A vector subspace of $\mathbb{R}^{n+1}_1$ is called {\it spacelike} if all of its nonzero vectors are spacelike; it is called {\it timelike} if it has a timelike vector and it is called {\it lightlike} if none of the above conditions are satisfied \cite{ratcliffe}. In Figure~1 we illustrate each case with planes in  $\mathbb{R}^{3}_1$ with respect to their position with respect to the {\it lightcone}
$$ LC = \{ x \in \mathbb{R}^{n+1}_1 \ : \ \left<x,x \right>=0\}. $$

\begin{figure}[h]  \label{fig:subspaces in 3d}
	\centering
	\includegraphics[scale=0.3]{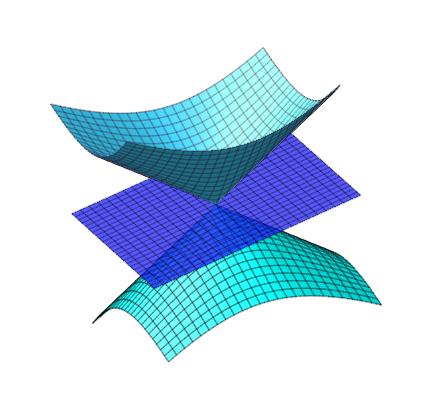}
	\includegraphics[scale=0.3]{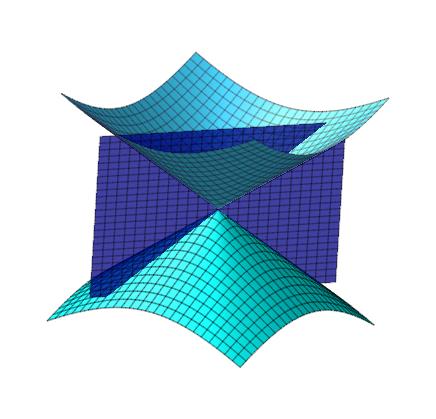}
	\includegraphics[scale=0.33]{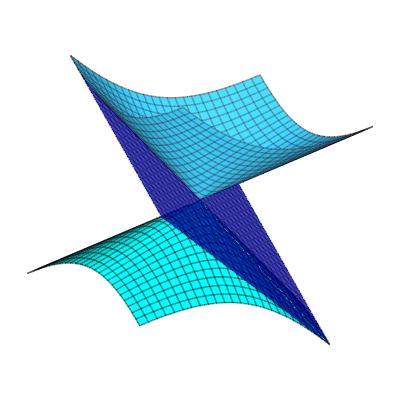}
	\caption{From left to right: a space-, a time- and a lightlike plane in $\mathbb{R}^{3}_1$.}
	\label{fig:subspacesin3d}
\end{figure}

We start with an important lemma whose proof is immediate:

\begin{lemma} \label{lemma:basic}
Let $W$ be a subspace of $\mathbb{R}^{n+1}_1$.

\noindent (a) For any $\gamma \in {\bf O}(n,1)$, $W$ and $\gamma W$ are subspaces of same type (space-, time- or lightlike)

\noindent  (b) If \  $\Sigma_1, \Sigma_2 < {\bf O}(n,1)$ are conjugate subgroups, namely $\Sigma_2 = \gamma^{-1} \Sigma_1 \gamma$, then $W$ is $\Sigma_1$-invariant if, and only if,  $\gamma W$ is $\Sigma_2$-invariant.
\end{lemma}

A subspace is {\it nondegenerate} if the pseudo inner product  restricted to it  is a nondegenerate bilinear form. It is well-known that $W$ is nondegenerate if, and only if, $W \cap W^{\perp}$ is trivial (see \cite{oneill1983}). We show that an invariant subspace admits $W^{\perp}$ as  invariant orthogonal complement if, and only if, it is nondegenerate (Proposition~\ref{complementonaodegenerado}). If the subspace is degenerate (lightlike) we shall see that it still  admits a complement, but this is nonorthogonal, and its invariance  is attained  under a condition imposed on the group  (Proposition~\ref{complementoluz}).

\begin{proposition} \label{complementonaodegenerado}
	For any subspace $W \subseteq \mathbb{R}^{n+1}_1$,
	\begin{itemize}
		\item[(a)] $\dim W+ \dim W^{\perp} =n+1$.
		\item[(b)] For a group $\Gamma < {\bf O}(n,1)$,  $W$ is nondegenerate and $\Gamma$-invariant if, and only if,  its orthogonal subspace is a $\Gamma$-invariant complement, namely   $W \oplus W^{\perp}=\mathbb{R}^{n+1}_1$.
	\end{itemize}
\end{proposition}

\begin{proof}
	\begin{itemize}
		\item[(a)] Consider an operator $\varphi : \mathbb{R}^{n+1}_1 \to \mathbb{R}^{n+1}_1$,  $\varphi (x)=AJx$, where $A$ is a square matrix whose lines are formed by basic vectors of $W$ and by lines with all entries equal to 0. We have that
		$\dim Im \varphi= rank(A)= \dim W$. Thus, $\dim \ker \varphi+ \dim W=n+1$. But  $\ker \varphi=W^{\perp}$, so the result follows.
		\item[(b)] The "if" part is immediate: for  $u \in W^{\perp}$ and $\gamma \in \Gamma$, from the $\Gamma$-invariance of $W$, for all $w \in W$ we have
		\begin{eqnarray*}
			\left<w, \gamma u \right> = \left<\gamma^{-1} w, u\right>=0.
		\end{eqnarray*}
		Now, $W$ is nondegenerate if, and only if,  $W \cap W^{\perp}$ is trivial which, together with (a), gives the result.
	\end{itemize}	
\end{proof}

\begin{proposition} \label{complementoluz}
	Let $\Gamma < {\bf O}(n,1)$ such that $\gamma^t \in \Gamma$, for all $\gamma \in \Gamma$. If $W$ is a $\Gamma$-invariant lightlike subspace, then $JW^{\perp}$ is a $\Gamma$-invariant complement of $W$, with $W \oplus JW^{\perp}=\mathbb{R}^{n+1}_1$.
\end{proposition}

\begin{proof}
First, notice that $JW^{\perp}$ is $\Gamma$-invariant:  for $u = Jv \in J W^{\perp}$,
$$\gamma u=\gamma Jv =J (\gamma^t)^{-1} v,$$
which belongs to $JW^{\perp}$, since $W^{\perp}$ is $\Gamma$-invariant and $\gamma^t \in \Gamma$ by hypothesis. Also, if  $w = Jv \in W \cap JW^{\perp}$,  then
	$\left< w,v \right>=0$, and so $\left< Jv,v \right>=0$, implying that $v=0$. Hence,  $ W \cap JW^{\perp}$ is trivial, Now, it follows from (a) of Proposition~\ref{complementonaodegenerado} that   $\dim W + \dim(J W^{\perp}) = n+1$. Therefore, $W \oplus JW^{\perp}=\mathbb{R}^{n+1}_1$.
\end{proof}

The two propositions above  provide the general way to decompose   $\mathbb{R}^{n+1}_1$ as  a direct sum of an invariant  subspace and its complement, depending on its type:  if $U$ is a spacelike or a timelike subspace and $W$ is lightlike subspace, then
$$ \mathbb{R}^{n+1}_1 \ = \  U \oplus U^\perp  \ = \  W \oplus J W^\perp. $$

In presence of symmetries, one important class of invariant subspaces is the class of fixed-point subspace. For a given group  $\Gamma < {\bf O}(n,1)$, Recall that the fixed-point subspace of a subgroup $\Sigma < \Gamma$ is the subspace
 $$ {\rm Fix}(\Sigma)=\{v \in V : \sigma v=v, \forall \sigma \in \Sigma\}.$$
 There are two  basic facts that are the main motivations for the results presented in this section, concerning the applications. First, we recall that  for a $\Gamma$-equivariant mapping  $g$,
 \begin{eqnarray} \label{eq:equiv}
 g( {\rm Fix}(\Sigma)) \subset  {\rm Fix}(\Sigma).
 \end{eqnarray}
 Also, fixed-point subspaces of conjugate subgroups are related by
 \begin{eqnarray} \label{eq:conjugacy}
 \gamma {\rm Fix}(\Sigma) \subseteq {\rm Fix}(\gamma \Sigma \gamma^{-1}).
 \end{eqnarray}
One interest is related to  symmetric dynamics: when this is ruled by a $\Gamma$-equivariant mapping  $g$ (defining a vector field, for instance), we have that (\ref{eq:equiv}) holds
for all $\Sigma < \Gamma$. Therefore, these are subspaces on which the dynamics must remain invariant.
In another direction, we mention the study of the geometry of surfaces which are given as the inverse image $f^{-1}(c)$, \ $c \in \mathbb{R}$, for some $f$ $\Gamma$-invariant function. By construction, the whole group leaves this surface setwise invariant and, in addition, the whole space is foliated by these surfaces in a symmetric way. Now, recall that the normalizer N($\Sigma$) is the symmetry group of the set $ {\rm Fix}(\Sigma)$, in the sense that it is the largest subgroup of $\Gamma$ that leaves $ {\rm Fix}(\Sigma)$ setwise invariant. Hence, we can use that structure to understand the surface "in peaces" preserving their symmetries, once $\Gamma  \backslash$ N($\Sigma$) are the symmetries of  $f^{-1}(0) \cap  {\rm Fix}(\Sigma)$. It now follows from (\ref{eq:conjugacy}) that each peace can be taken in conjugacy classes of subgroups of $\Gamma$.

Notice that in $\mathbb{R}^{3}_1$  the orthogonal subspace to a lightlike line $W$ is a plane tangent  to the lightcone which contains  $W$. Likewise, if $W$ is a lightlike plane then its orthogonal subspace is a lightlike line contained in $W$. This follows directly from the fact that if $W$ in $\mathbb{R}^{n+1}_1$ is lightlike then $\dim W \cap W^{\perp}=1$ (see \cite{izumiya2009legendrian} for example) and also from the fact that $W$ is spacelike if, and only if, $W^{\perp}$ is timelike (see  \cite{ratcliffe} for example).

\subsection{Invariant lines in $\mathbb{R}^{2}_1$} \label{subsec:2d}

The standard 2-dimensional representation of the Lorentz group is
\begin{eqnarray} \label{eq:O11}
 {\bf O}(1,1) = {\bf SO}_0(1,1) \ \dot{\cup} \  \Lambda^{p} {\bf SO}_{0}(1,1) \  \dot{\cup} \ \Lambda^{t} {\bf SO}_{0}(1,1) \  \dot{\cup}  \  - I {\bf SO}_{0}(1,1),
 \end{eqnarray}
 where ${\bf SO}_0(1,1)$ is the group of {\it hyperbolic rotations,}
 \begin{eqnarray} \label{eq:hyper rot}
 H_\theta :=  \left( \begin{array}{cc}    \cosh \theta &  \sinh \theta \\  \sinh \theta &    \cosh \theta \end{array} \right),
  \end{eqnarray}
for $\theta \in {\mathbb{R}}$, and
$$ \Lambda^{p} =\left( \begin{array}{cc} -1 & 0 \\ 0 & 1 \end{array} \right), \ \  \Lambda^{t}=\left( \begin{array}{cc} 1 & 0 \\ 0 & -1 \end{array} \right).$$

\begin{itemize}
\item For an arbitrary element  in  $ {\bf SO}_0(1,1)$ or in the component  $ -I {\bf SO}_0(1,1)$,  it is  direct that  the invariant lines under this element are the two light lines of the lightcone,
$$W_1 = \{ (x,y) \in\mathbb{R}^{2}_1 \ :  y=x\}, \ \ \ W_2 = \{ (x,y) \in\mathbb{R}^{2}_1 \ :  y=-x\}. $$
These are degenerate subspaces and one is the complement of the other as in Proposition~\ref{complementoluz}. Clearly neither is a fixed-point subspace.

\item The invariant lines under each $\Lambda^p$ and $\Lambda^t$ are the $x$-axis (a space line) and the $y$-axis (a time line), and one is the orthogonal complement of the other. The $x$-axis is ${\rm Fix}({\bf Z}_2(\Lambda^t))$  and the $y$-axis is ${\rm Fix}({\bf Z}_2(\Lambda^p))$.  Any other involution belongs to one of the other two components of (\ref{eq:O11}) and it is conjugate to   either $\Lambda^p$, if it is of the form  $\Lambda^p H_\theta$,  or to $\Lambda^t$, if it is of the form $\Lambda^t H_\theta$, the conjugacy matrix for both cases being $\gamma =  H_{-\theta/2}$.
Up to conjugacy, there are therefore   only two classes of involutions.
For any $\theta \in {\mathbb R}$, we can use Lemma~\ref{lemma:basic} for this $\gamma$ to obtain explicitly all the invariant lines under the involutions of ${\bf O}(1,1)$.

\item Any subgroup $\Sigma$ generated by two or more elements has nontrivial invariant subspaces only if $\Sigma < {\bf SO}_0(1,1) \times {\bf Z}_2(-I)$, and these are $W_1$ and $W_2$ above.
\end{itemize}
In Table~1 we  summarize these results. The straight lines in the second column are the lines invariant under the subgroups  given in the first colum. Their type is given in the third column and the last column shows their complement subspaces.

\begin{table}[!ht] \label{table1}
	\footnotesize
	\begin{center}
		\caption{Invariant subspaces of $\mathbb{R}^2_1$.}
		\begin{tabular}{c|c|c|c}
			\hline \hline
			Subgroup of ${\bf O}(1,1)$  & Invariant subspaces & Type &  Complement subspace \\ \hline \hline
			$[H_{\theta}]$ & $\begin{array}{l}   W_1 = \{(x,y) \in \mathbb{R}^{2}_1 : y= x \}    \\  W_2 =   \{(x,y) \in \mathbb{R}^{2}_1 : y=- x \}  \end{array}$ &  {\rm light} &
			$\begin{array}{c}  J W_1^\perp  = W_2 \\ J W_2^\perp = W_1 \end{array}$ \\ \hline
	
          $[- H_{\theta}]$ & $\begin{array}{c}   W_1   \\ W_2   \end{array}$ & {\rm light} &
			$\begin{array}{c}  J W_1^\perp  = W_2 \\ J W_2^\perp = W_1 \end{array}$  \\ \hline
		
		${\bf Z}_2( \Lambda^{t})$ & $\begin{array}{c}     {\rm Fix}({\bf Z}_2(\Lambda^t)) =  x{\rm- axis}\\ y{\rm-axis} \end{array}$ & $\begin{array}{c}  {\rm space} \\ {\rm time} \end{array}$ &
					$\begin{array}{c} {\rm Fix}({\bf Z}_2(\Lambda^t))^\perp \\ {\rm Fix}({\bf Z}_2(\Lambda^p))^\perp \end{array}$ \\ \hline	
			
				${\bf Z}_2( \Lambda^{p})$ & $\begin{array}{c}   {\rm Fix}({\bf Z}_2(\Lambda^p)) = y{\rm-axis}  \\ x{\rm -axis} \end{array}$ & $\begin{array}{c}  {\rm time} \\ {\rm space} \end{array}$ &
									$\begin{array}{c} y{\rm -axis} \\ x{\rm-axis} \end{array}$ \\ \hline	
		\end{tabular}
	\end{center}
\end{table}

\subsection{Invariant lines and invariant planes in $\mathbb{R}^{3}_1$} \label{subsec:3d}

The standard 3-dimensional representation of the Lorentz group is
\begin{eqnarray} \label{eq:O21}
 {\bf O}(2,1) = {\bf SO}_0(2,1) \ \dot{\cup} \  \Lambda^{p} {\bf SO}_{0}(2,1) \  \dot{\cup} \ \Lambda^{t} {\bf SO}_{0}(2,1) \  \dot{\cup}  \  \Lambda^{pt} {\bf SO}_{0}(2,1),
 \end{eqnarray}
 where ${\bf SO}_0(2,1)$ is the group of {\it hyperbolic rotations} $H^{+}$ which, using  singular value decomposition    (\cite{Gallier}), are written as
 	\begin{eqnarray} \label{eq:3d hyp rot}
H^{+} =
		\left (\begin{array}{ccc}  \cos \varphi & -\sin \varphi  & 0 \\ \sin \varphi  &  \cos \varphi & 0 \\ 0 & 0 & \epsilon \end{array} \right)
		\left (\begin{array}{ccc} 1  & 0 & 0 \\
			0 &     \cosh \theta &  \sinh \theta \\
			0 &  \sinh \theta &     \cosh \theta
		\end{array} \right)
		\left (\begin{array}{ccc}  \cos \phi & -\sin(\phi)  & 0 \\ \sin(\phi) &  \cos \phi & 0 \\ 0 & 0 & 1 \end{array} \right),
	\end{eqnarray}
for	 $\varphi, \theta, \phi \in {\Bbb R}$ and $\epsilon = 1$, and
\begin{eqnarray} \label{eq:lambdas}
\Lambda^p = \left (\begin{array}{ccc}  1 & 0  & 0 \\ 0 & -1 & 0 \\ 0 & 0 & 1 \end{array} \right), \ \ \Lambda^t = J.
\end{eqnarray}
Matrices in the component 	 $\Lambda^{pt} {\bf SO}_{0}(2,1)$ are of the form
	\begin{eqnarray} \label{eq:3d involutions}
H^{-} =
		\left (\begin{array}{ccc}  \cos \varphi & \sin \varphi  & 0 \\ \sin \varphi  & - \cos \varphi & 0 \\ 0 & 0 & \epsilon \end{array} \right)
		\left (\begin{array}{ccc} 1  & 0 & 0 \\
			0 &     \cosh \theta &  \sinh \theta \\
			0 &  \sinh \theta &   \epsilon  \cosh \theta
		\end{array} \right)
		\left (\begin{array}{ccc}  \cos \phi & -\sin \phi   & 0 \\ \sin(\phi) &  \cos \phi & 0 \\ 0 & 0 & 1 \end{array} \right),
	\end{eqnarray}
for	 $\varphi, \theta, \phi \in {\Bbb R}$ and $\epsilon = -1$. Matrices in 	 $\Lambda^t  {\bf SO}_{0}(2,1)$ are of the form (\ref{eq:3d hyp rot}) for $\epsilon = -1$ and in 	 $\Lambda^p  {\bf SO}_{0}(2,1)$ are of the form (\ref{eq:3d involutions}) for $\epsilon = 1$.

Below we consider invariant subspaces under an element in each of the connected components of ${\bf O}(2,1)$. We recall that elements in distinct components are not conjugate.
	
	 \begin{itemize}
	 \item For  $H^{+} \in {\bf SO}_0(2,1)$,  given in  (\ref{eq:3d hyp rot}) with $\epsilon = 1$,  if  $\varphi + \phi = \pi$ then it is conjugate to $-\Lambda^t$, by the conjugacy matrix 
	\begin{eqnarray} \label{matrizconj}
	 \left( \begin{array}{ccc} 
	 	 	-\cos \phi  & -\sin \phi\cosh \theta/2 & -\sin \phi \sinh \theta/2 \\
	 	 	\sin \phi & -\cos \phi \cosh \theta/2 & -\cos \phi \sinh \theta/2 \\
	 	 	0 & \sinh \theta/2 & \cosh \theta/2
	 	 	\end{array} \right).
	 	 	\end{eqnarray} 
The invariant lines under $-\Lambda^t$ are the time line given by the $z$-axis, which is Fix(${\bf Z}_2(- \Lambda^t)$), and the space lines in the plane $z=0$. The invariant planes are the space plane $z=0$ and all the time planes containing the $z$-axis.  	
	
	 \item  For $H^{+} \in  {\bf SO}_0(2,1)$ given in (\ref{eq:3d hyp rot}) with $\epsilon = 1$, if  $\varphi + \phi  \neq \pi $, invariant lines and planes can be of any type, depending on the values of $\varphi, \theta, \phi$.  The subgroup $[H^{+}]$ generated by $H^{+}$ is noncompact, and  Fix($[H^{+}]$) is the line generated by the vector
	 $$  \Bigl(1,\frac{\sin\varphi - \sin \phi}{ \cos \phi+ \cos \varphi}, \frac{( \cos \varphi\sin \phi+\sin \varphi  \cos \phi \sinh \theta}{(  1- \cosh \theta)( \cos \phi+ \cos \varphi)}\Bigr), $$
which can also be of any type.  	
	
	 \item  For $\Lambda^{t} H^{+} \in \Lambda^t {\bf SO}_0(2,1)$,  given in  (\ref{eq:3d hyp rot}) with $\epsilon = -1$,  if  $\varphi = - \phi $ then it is conjugate to $\Lambda^t$, by the conjugacy matrix (\ref{matrizconj}).  The invariant lines under $\Lambda^t$ are the time line given by the $z$-axis, and the space lines in the plane $z=0$. The invariant planes are the space plane $z=0$ which is Fix(${\bf Z}_2(\Lambda^t)$)  and all the time planes containing the $z$-axis.  	
	 	 	
	 \item  For $\Lambda^{t}  H^{+} \in  \Lambda^{t} {\bf SO}_0(2,1)$ given in (\ref{eq:3d hyp rot}) with $\epsilon = -1$ and for $\varphi \neq \phi$,  Fix(${\bf Z}_2(\Lambda^t H^{+})$) is trivial. There are no invariant lines or planes by the action of the group ${\bf Z}_2(\Lambda^t H^{+})$.
	
	 \item  For $\Lambda^{p} H^{-} \in \Lambda^p {\bf SO}_0(2,1)$,  given in  (\ref{eq:3d involutions}) with $\epsilon = 1$,  if  $\varphi = - \phi $ then it is conjugate to $\Lambda^p$, by the conjugacy matrix (\ref{matrizconj}).  The invariant lines under $\Lambda^p$ are the space line given by the $y$-axis, and the time lines in the plane $y=0$. AQUI The invariant planes are the time plane $y=0$ which is Fix(${\bf Z}_2(\Lambda^p)$)  and all the time planes containing the $y$-axis.  	
	 		 	 	
	\item  For $\Lambda^{p}  H^{-} \in  \Lambda^{p} {\bf SO}_0(2,1)$ given in (\ref{eq:3d hyp rot}) with $\epsilon = 1$, if  $\varphi \neq \phi$ then Fix(${\bf Z}_2(\Lambda^p H^{+})$) is trivial.  There are no invariant lines or planes by the action of the group ${\bf Z}_2(\Lambda^p H^{-})$.
	
	 \item For $H^{-} \in \Lambda^{pt}  {\bf SO}_0(2,1)$, given in  (\ref{eq:3d involutions}) with $\epsilon = -1$, if  $\varphi + \phi = \pi$ then it is conjugate to $-\Lambda^p$, by the conjugacy matrix (\ref{matrizconj}). The invariant lines under $-\Lambda^p$ are the space line given by the $y$-axis, which is Fix(${\bf Z}_2(-\Lambda^p)$), and the all lines in the plane $y=0$, whose types are time-, light- and spacelike.  The invariant planes are the time plane $y=0$ and all the planes containing the $y$-axis, which are the two light planes  $z \pm x = 0$ (tangent to the lightcone)  and all space and time planes containing the $y$-axis.  	
	 	 	
	 \item For $H^{-} \in \Lambda^{pt}  {\bf SO}_0(2,1)$, given in  (\ref{eq:3d involutions}) with $\epsilon = -1$, if  $\varphi + \phi \neq \pi$, invariant lines and planes can be of any type, depending on the values of $\varphi, \theta, \phi$.  The subgroup $[\Lambda^{pt} H^{+}]$ generated by $\Lambda^{pt} H^{+}$ is noncompact, and  Fix($[\Lambda^{pt} H^{+}]$) is the line generated by
	 	 $$ \Bigl(1,-\frac{\sin(\phi)-\sin(\varphi)}{ \cos \phi+ \cos \varphi}, -\frac{( \cos \varphi\sin(\phi)+\sin(\varphi) \cos \phi) \sinh \theta}{(   \cosh \theta+1)( \cos \phi+ \cos \varphi)}\Bigr) $$
	 which can also be of any type.
	 \end{itemize}	
Finally, if a  subgroup  $\Sigma < {\bf O}(2,1)$ is generated by more than one element, then these appear in the list above  and so Fix($\Sigma$) is computed directly as the intersection of the fixed-point subspaces of each of its generators. \\

\vspace*{1cm}
\noindent {\bf Acknowledgments.} The authors acknowledge partial support by CAPES under CAPES/FCT grant 99999.008523/2014-02  and  CAPES/PROEX  grant 1183747.

\end{document}